\documentclass[11pt,a4paper]{article}
\usepackage[T1]{fontenc}
\usepackage{lmodern}
\usepackage[margin=28mm]{geometry}
\usepackage{amsmath,amssymb,amsthm,mathtools}
\usepackage{microtype}
\usepackage{xcolor}
\usepackage{hyperref}
\hypersetup{pdftitle={Spectral isometries of unital $C^*$-algebras},pdfauthor={Matthew J. Colbrook and George Stepaniants},colorlinks=true,linkcolor=blue!40!black,citecolor=blue!40!black,urlcolor=blue!40!black}
\newtheorem{theorem}{Theorem}[section]
\newtheorem{proposition}[theorem]{Proposition}
\newtheorem{lemma}[theorem]{Lemma}
\newtheorem{corollary}[theorem]{Corollary}
\newcommand{\C}{\mathbb C}
\newcommand{\R}{\mathbb R}
\newcommand{\ad}{\operatorname{ad}}
\newcommand{\Ad}{\operatorname{Ad}}
\newcommand{\conv}{\operatorname{conv}}
\newcommand{\rad}{\operatorname{rad}}
\newcommand{\dist}{\operatorname{dist}}
\newcommand{\Rea}{\operatorname{Re}}
\newcommand{\Ima}{\operatorname{Im}}
\newcommand{\norm}[1]{\left\lVert #1\right\rVert}
\newcommand{\spec}{\sigma}

\title{Spectral isometries of unital $C^*$-algebras}
\author{Matthew J. Colbrook\thanks{Corresponding author. Email: \href{mailto:m.colbrook@damtp.cam.ac.uk}{\texttt{m.colbrook@damtp.cam.ac.uk}}.}
\quad George Stepaniants\thanks{Also at the Courant Institute, New York University, New York, USA. Email: \href{mailto:gs852@cam.ac.uk}{\texttt{gs852@cam.ac.uk}}.}\\[0.5em]
\small Department of Applied Mathematics and Theoretical Physics\\
\small University of Cambridge, Wilberforce Road, Cambridge CB3 0WA, United Kingdom}
\date{}

\begin{document}
\maketitle

\begin{abstract}
This paper provides a human digestion of a proof of the spectral isometry conjecture for $C^*$-algebras found by GPT-6 Astra. We prove that every unital surjective complex-linear spectral isometry between unital $C^*$-algebras is a Jordan isomorphism. Without the assumption that the map preserves the identity, it is a central unitary multiple of a Jordan isomorphism. The main step is to show that an entire one-parameter group of unital spectral isometries consists of algebra automorphisms. To prove this, we extend the group to the bidual, which splits into finite and properly infinite summands. On the finite part, we use quadratic and cubic trace identities and on the properly infinite part, we apply a structure theorem for maps that preserve square-zero elements. To show that both summands are invariant, we use holomorphic dependence on the parameter together with a uniform resolvent bound, which ensures that self-adjoint elements have images with real spectrum. We then apply the group theorem to inner automorphisms conjugated by the original spectral isometry. This reduces the problem to the commutative case, where Nagasawa’s theorem completes the proof.
\end{abstract}

\section{Introduction}
For an element $a$ of a unital complex Banach algebra $A$, let
\[
 r_A(a)=\max\{|\lambda|:\lambda\in\spec_A(a)\}
\]
denote its spectral radius. A complex-linear map $T:A\to B$ is a \emph{spectral isometry} if $r_B(T(a))=r_A(a)$ for every $a\in A$. The spectral isometry conjecture asks whether every unital surjective spectral isometry between unital $C^*$-algebras is a Jordan isomorphism; see \cite[Conjecture 4.6]{MathieuSchick2002} and \cite[Problem 5]{Mathieu2009}. Here, a Jordan isomorphism is a complex-linear bijection that preserves squares.

We prove the conjecture for arbitrary unital $C^*$-algebras and the result has been formally verified in Lean \cite{palomar-2026-10-28-000031-v1}.

\begin{theorem}\label{thm:main}
Let $A$ and $B$ be unital complex $C^*$-algebras, and let $T:A\to B$ be a surjective complex-linear map such that
\[
 T(1_A)=1_B,\qquad r_B(T(a))=r_A(a)\qquad(a\in A).
\]
Then $T$ is a Jordan isomorphism. In particular,
\[
 T(a^2)=T(a)^2\qquad(a\in A).
\]
\end{theorem}

Kadison's theorem identifies unital surjective complex-linear norm isometries between $C^*$-algebras as Jordan $*$-isomorphisms \cite[Theorem 7]{Kadison1951}. Theorem~\ref{thm:main} shows that $T$ preserves squares, but not necessarily adjoints. For example, the map $a\mapsto sas^{-1}$, with $s$ invertible, preserves products and spectra but need not preserve adjoints. The commutative case was already known due to the spectral-radius form of Nagasawa's theorem \cite[Corollary 2.5]{MathieuYoung} which shows that $T$ is multiplicative, extending the earlier result for uniform algebras \cite{Nagasawa}. Previous work also established the conjecture for properly infinite von Neumann algebras \cite{MathieuSchick} and, more generally, unital $C^*$-algebras of real rank zero without tracial states \cite{Mathieu2020}.

Our main contribution is a theorem on entire one-parameter groups of spectral isometries: if $\delta$ is a bounded complex-linear operator on a unital $C^*$-algebra and $e^{z\delta}$ is a unital spectral isometry for every $z\in\C$, then $\delta$ is a derivation and every $e^{z\delta}$ is an algebra automorphism (Proposition~\ref{prop:entire}). The proof treats the finite and properly infinite parts of the bidual separately. A uniform resolvent estimate shows that the extended maps send self-adjoint elements to elements with real spectrum(Lemma~\ref{lem:resolvent}), while an approximation argument shows that they also preserve square-zero elements.

Using traces on the finite summand, we obtain quadratic and cubic identities from Liouville's theorem. These identities show that the properly infinite summand is invariant, allowing us to treat the finite part as a quotient. The argument that the tracial identities force preservation of squares comes from Aupetit's trace method \cite{Aupetit1998}, in the form developed by Mathieu and Schulz \cite{MathieuSchulz2023}. On the properly infinite summand, we use the weighted Jordan structure of surjective square-zero preservers \cite{Mathieu2009}. Preservation of real spectra and holomorphy force the central weight to be the identity and the off-diagonal component to vanish. This gives a group of Jordan automorphisms of the bidual; differentiation and Sinclair's theorem \cite{Sinclair} give the derivation assertion.

To prove Theorem~\ref{thm:main}, we conjugate inner automorphisms by $T$. Such holomorphic similarity families were used by Costara and Repov\v{s} \cite[Lemma 3]{CostaraRepovs} and Mathieu and Sourour \cite[Lemma 2.4]{MathieuSourour}, under additional assumptions on finite-dimensional representations or quotients. Proposition~\ref{prop:entire} shows that the generators of these families are derivations for arbitrary unital $C^*$-algebras. For each self-adjoint $h$, we write $T(\{h\}'')$ as an intersection of kernels of these derivations, so it is a subalgebra. The bicommutant criterion of Mathieu and Sourour \cite[Proposition 3.1]{MathieuSourour} allows us to reduce the proof to the commutative case.

\paragraph*{AI declaration}

The proof was found by GPT-6 Astra after a sequence of simple prompts. We then wrote the paper after discussing the proof further with ChatGPT 6 to understand the underlying mechanisms. After we wrote the paper ourselves, ChatGPT 6 was used for minor language editing of the penultimate draft in this stage only.

\section{Preliminaries}
All algebras and linear maps are complex. Identities are denoted by $1$ when the algebra is clear. We identify a $C^*$-algebra with its canonical image in its bidual. For a unital algebra $E$, $Z(E)$ denotes its centre and $\rad(E)$ its Jacobson radical. For a compact subset $K$ of $\C$, $\conv K$ denotes its convex hull, which is compact. We omit algebra subscripts on spectral radii and spectra when there is no ambiguity.

Every surjective spectral isometry $T:A\to B$ between $C^*$-algebras is injective. Indeed, if $T(a)=0$, then $r(a)=0$ and
\[
 r(a+a^*)=r\bigl(T(a+a^*)\bigr)=r\bigl(T(a^*)\bigr)=r(a^*)=0.
\]
Since $a+a^*$ is self-adjoint, it vanishes. Thus $a$ is skew-adjoint and hence normal, so $\norm{a}=r(a)=0$. Aupetit's automatic-continuity theorem \cite{Aupetit1982}, in the form stated in \cite[Theorem 2.1]{Mathieu2009}, and the bounded inverse theorem imply that $T$ and $T^{-1}$ are bounded. The case of the zero algebra in Theorem~\ref{thm:main} is immediate; henceforth the algebras are nonzero.

We use Kaplansky density, joint strong$^*$ continuity of multiplication on bounded sets, and the decomposition of a von Neumann algebra into finite and properly infinite central summands. A finite von Neumann algebra has a faithful normal centre-valued trace \cite[Theorem III.2.5.7 and Corollary III.2.5.8]{Blackadar}. Composing this trace with normal states of the centre gives normal tracial states that separate positive elements.

The strong$^*$ topology on a von Neumann algebra $M$ is generated by the seminorms $x\mapsto\bigl(\varphi(x^*x)+\varphi(xx^*)\bigr)^{1/2}$, where $\varphi\in M_*^+$.

\begin{lemma}\label{lem:normal-continuity}
A bounded weak$^*$-continuous linear map between von Neumann algebras is strong$^*$-continuous on norm-bounded sets.
\end{lemma}
\begin{proof}
This follows from Akemann's continuity theorem \cite[Corollary III.10]{Akemann}, first on the unit ball and then by rescaling. The bounded preadjoint maps weakly compact sets to weakly compact sets, giving continuity for the Mackey topologies. These coincide with the strong$^*$ topologies on bounded sets \cite[Theorem II.7]{Akemann}.
\end{proof}

The next elementary formula converts spectral information about perturbations into a resolvent bound. For a unital algebra $D$, let $D^{-1}$ denote its group of invertible elements.

\begin{lemma}\label{lem:distance}
If $c$ is invertible in a unital $C^*$-algebra $D$, then
\[
 \dist(c,D\setminus D^{-1})=\norm{c^{-1}}^{-1}.
\]
\end{lemma}
\begin{proof}
The Neumann series gives the lower bound. For the reverse bound, let $c=v|c|$ be the polar decomposition of $c$ and let
$
 m=\min\spec(|c|)=\norm{c^{-1}}^{-1}.
$
The element $v$ is unitary, and $c-mv=v(|c|-m1)$ is noninvertible. Its distance from $c$ is $m$.
\end{proof}

\section{Passage to the bidual}\label{sec:bidual}
Throughout this section, $U:C\to D$ is a unital bijective spectral isometry between unital $C^*$-algebras. We establish two properties of $U^{**}$ used in the next section.

\subsection{Real spectra and a resolvent estimate}
Preservation of convex hulls of spectra is due to Mathieu and Schick \cite[Proposition 4.8]{MathieuSchick2002}. Its short proof will also be useful for the resolvent estimate.

\begin{lemma}\label{lem:hull}
For every $x\in C$,
\[
 \conv\spec_D(U(x))=\conv\spec_C(x).
\]
Consequently, $U$ sends self-adjoint elements to elements with real spectrum.
\end{lemma}
\begin{proof}
Unitality and complex linearity give
\[
 r_D(U(x)-\lambda1)=r_C(x-\lambda1)\qquad(\lambda\in\C).
\]
For a nonempty compact set $K\subset\C$, define $d_K(\lambda)=\max_{w\in K}|w-\lambda|$. For every $\theta\in\R$, the support function of $K$ satisfies
\[
 \max_{w\in K}\Rea(e^{-i\theta}w)
 =\lim_{t\to\infty}\bigl(d_K(-te^{i\theta})-t\bigr).
\]
Thus $d_K$ determines $\conv K$. Applying this observation to the two spectra proves the identity. If $x=x^*$, its spectrum and hence the spectrum of $U(x)$ lie in $\R$.
\end{proof}

To pass this conclusion to the bidual, we need uniform control of the resolvents of images of self-adjoint elements.

\begin{lemma}\label{lem:resolvent}
Let $K=\norm{U^{-1}}$. For $a=a^*\in C$ and $\lambda\in\C\setminus\R$,
\begin{equation}\label{eq:resolvent}
 \norm{(U(a)-\lambda1)^{-1}}\le\frac{K}{|\Ima\lambda|}.
\end{equation}
Moreover,
\begin{equation}\label{eq:bidual-real}
 a=a^*\in C^{**}\quad\rightarrow\quad
 \spec_{D^{**}}(U^{**}(a))\subset\R.
\end{equation}
\end{lemma}
\begin{proof}
Suppose that $e\in D$ and $\norm{e}<|\Ima\lambda|/K$, and let $h=U^{-1}(e)$. The self-adjoint resolvent estimate and the Neumann series imply that
\[
 \spec_C(a+h)\subset
 \{\zeta\in\C:|\Ima\zeta|\le\norm{h}\}.
\]
This strip is convex and does not contain $\lambda$. Lemma~\ref{lem:hull} therefore shows that $U(a)+e-\lambda1$ is invertible. Every noninvertible element is consequently at distance at least $|\Ima\lambda|/K$ from $U(a)-\lambda1$. Lemma~\ref{lem:distance} gives \eqref{eq:resolvent}.

Now let $a=a^*\in C^{**}$. By Kaplansky density, there is a bounded self-adjoint net $(a_i)$ in $C$ converging to $a$ in the strong$^*$ topology. Since $U^{**}$ is weak$^*$-continuous, Lemma~\ref{lem:normal-continuity} gives $U(a_i)\to U^{**}(a)$ in the strong$^*$ topology. Fix $\lambda\notin\R$ and let $L=K/|\Ima\lambda|$. The elements $c_i=U(a_i)-\lambda1$ satisfy
\[
 c_i^*c_i\ge L^{-2}1,\qquad c_ic_i^*\ge L^{-2}1.
\]
These inequalities pass to the strong$^*$ limit $c=U^{**}(a)-\lambda1$. The first gives a left inverse for $c$ and the second gives a right inverse. Hence $c$ is invertible, proving \eqref{eq:bidual-real}.
\end{proof}

\subsection{Square-zero elements}
Square-zero preservation for spectral isometries between $C^*$-algebras is known \cite[Lemma 3.1 and Corollary 3.2]{MathieuSchick}. We give a direct proof and then extend the property to the bidual.

\begin{lemma}\label{lem:squarezero}
If $n\in C$ and $n^2=0$, then $U(n)^2=0$.
\end{lemma}
\begin{proof}
Let $b=U(n)$, $K=\norm{U^{-1}}$, and $H=2\norm{n}K+K^2$. Unitality gives $K\ge1$ and hence $H\ge1$. For $e\in D$ with $\norm{e}\le1$, let $h=U^{-1}(e)$. Then
$$
 r(b+e)^2
 =r(n+h)^2=r\bigl((n+h)^2\bigr)
 \le2\norm{n}\norm{h}+\norm{h}^2
 \le H\norm{e}.
$$
Also $r(b)=0$. If $0<|\lambda|\le1$ and $\norm{e}<|\lambda|^2/H$, then $r(b+e)<|\lambda|$. Thus $\lambda1-b-e$ is invertible. Lemma~\ref{lem:distance} yields
\[
 \norm{(\lambda1-b)^{-1}}\le H|\lambda|^{-2}.
\]
Since $\spec(b)=\{0\}$, the holomorphic functional calculus gives
\[
 b^2=\frac{1}{2\pi i}\int_{|\lambda|=\varepsilon}
 \lambda^2(\lambda1-b)^{-1}\,d\lambda
 \qquad(0<\varepsilon\le1).
\]
The contour is positively oriented. It follows that $\norm{b^2}\le H\varepsilon$, so $b^2=0$.
\end{proof}

The approximation in the next lemma is also a consequence of square-zero lifting \cite{AkemannPedersen1977} and, more generally, nilpotent-contraction lifting \cite{Shulman2008}. To see this, consider the $C^*$-algebra of bounded strong$^*$-convergent $C$-valued nets, indexed by the strong$^*$ neighbourhoods of zero in $C^{**}$ ordered by reverse inclusion. Kaplansky density makes the limit map onto $C^{**}$ a surjective $*$-homomorphism. Lifting a square-zero contraction gives contractive square-zero approximants. We give a direct proof.

\begin{lemma}\label{lem:approximation}
Every square-zero element of $C^{**}$ is a bounded strong$^*$ limit of square-zero elements of $C$. Consequently, square-zero preservation passes to the bidual of any bounded linear map between $C^*$-algebras.
\end{lemma}
\begin{proof}
For a positive element $x\in C^{**}$, let $s(x)$ denote its support projection. Let $n\in C^{**}$ satisfy $n^2=0$, and let $p=s(nn^*)$ and $q=s(n^*n)$ be its left and right support projections. The identity $(n^*n)(nn^*)=0$ shows that $p$ and $q$ are orthogonal, and $n=pnq$.

For a self-adjoint element $t$, let $t_+=(|t|+t)/2$. Choose self-adjoint contractions $h_i\in C$ converging strong$^*$ to $p-q$ and a bounded net $y_j\in C$ converging strong$^*$ to $n$. Continuous functional calculus on bounded self-adjoint sets gives
\[
 e_i=(h_i)_+\rightarrow p,\qquad
 f_i=(-h_i)_+\rightarrow q
\]
in the strong$^*$ topology. Moreover, $e_if_i=f_ie_i=0$. On the product directed set, the elements
\[
 n_{i,j}=e_i y_j f_i
\]
are uniformly bounded, have square zero, and converge strong$^*$ to $pnq=n$.

If $V:C\to D$ is bounded and preserves square-zero elements, Lemma~\ref{lem:normal-continuity} gives $V(n_{i,j})\to V^{**}(n)$ strong$^*$. Joint strong$^*$ continuity of multiplication on bounded sets implies that $V^{**}(n)^2=0$.
\end{proof}

\section{Entire groups of spectral isometries}
The properties established in Section~\ref{sec:bidual} allow us to treat the finite and properly infinite summands of the bidual separately. We first need the standard spectral bound for tracial states, recalled in \cite[Section 2]{MathieuSchulz2023}.

\begin{lemma}\label{lem:tracebound}
If $\tau$ is a tracial state on a unital $C^*$-algebra $C$, then
\[
 |\tau(c)|\le r(c)\qquad(c\in C).
\]
\end{lemma}
\begin{proof}
Fix $t>r(c)$. The positive series
\[
 h=\sum_{n=0}^{\infty}t^{-2n}(c^*)^n c^n
\]
converges in norm by the spectral radius formula. It satisfies $h\ge1$ and $c^*hc=t^2(h-1)$. Hence
\[
 (h^{1/2}ch^{-1/2})^*(h^{1/2}ch^{-1/2})
 =t^2(1-h^{-1})\le t^2 1.
\]
Since $\tau$ is invariant under similarities,
\[
 |\tau(c)|=|\tau(h^{1/2}ch^{-1/2})|\le t.
\]
Letting $t\downarrow r(c)$ proves the result.
\end{proof}

\begin{proposition}\label{prop:entire}
Let $C$ be a unital $C^*$-algebra and let $\delta:C\to C$ be a bounded complex-linear operator. Suppose that $U_z=e^{z\delta}$ is a unital spectral isometry for every $z\in\C$. Then $\delta$ is a derivation and each $U_z$ is an algebra automorphism.
\end{proposition}
\begin{proof}
The group law gives $U_z^{-1}=U_{-z}$. Let $M=C^{**}$ and $V_z=U_z^{**}=e^{z\delta^{**}}$. Each $V_z$ is a bounded weak$^*$-continuous bijection. By Section~\ref{sec:bidual}, it preserves square-zero elements and sends self-adjoint elements to elements with real spectrum. We do not assume that $V_z$ preserves spectral radii on $M$. We first prove that each $V_z$ is a Jordan isomorphism. Decompose
\[
 M=P\oplus F,\qquad P=pM,\qquad F=qM,\qquad p+q=1,
\]
where $p,q$ are orthogonal central projections, $P$ is properly infinite, and $F$ is finite. Arguments involving a zero summand are omitted.

\emph{Tracial identities.}
Let $\tau$ be a normal tracial state of $F$, extended by zero on $P$. Since $\tau(q)=1$, its restriction to $C$ is a tracial state. For $b\in C$ and $n=2,3$, the function
\[
 z\mapsto\tau\bigl(U_z(b)^n\bigr)
\]
is entire. Lemma~\ref{lem:tracebound} bounds its absolute value by $r(U_z(b)^n)=r(b)^n$. Liouville's theorem therefore gives
\[
 \tau\bigl(U_z(b)^n\bigr)=\tau(b^n).
\]
For fixed $z$ and $n$, Kaplansky density and Lemma~\ref{lem:normal-continuity} extend this identity to $M$. Joint strong$^*$ continuity of multiplication on bounded sets and normality of $\tau$ give
\begin{equation}\label{eq:moments}
 \tau\bigl(V_z(b)^n\bigr)=\tau(b^n)
 \qquad(b\in M;\ n=2,3).
\end{equation}
Polarising the quadratic identity and comparing the coefficient of $t$ in the cubic identity for $b+ta$, using the trace property, yields
\begin{align}
 \tau\bigl(V_z(a)V_z(b)\bigr)&=\tau(ab),\label{eq:bilinear}\\
 \tau\bigl(V_z(a)V_z(b)^2\bigr)&=\tau(ab^2).
 \label{eq:cubic}
\end{align}
The normal tracial states of $F$ separate positive elements. In particular,
\[
 P=\{x\in M:\tau(x^*x)=0\text{ for every such }\tau\}.
\]
If $x\in P$, then \eqref{eq:bilinear} gives $\tau(V_z(x)V_z(y))=0$ for every $y\in M$. Since $V_z$ is surjective, choose $y\in M$ such that $V_z(y)=V_z(x)^*$. Thus $\tau(V_z(x)V_z(x)^*)=0$ for every $\tau$, and separation implies $qV_z(x)=0$. Applying the same argument to $V_{-z}$ proves
\begin{equation}\label{eq:Pinvariant}
 V_z(P)=P.
\end{equation}

Relative to $M=P\oplus F$, define
\begin{equation}\label{eq:blocks}
 V_z(a+f)=R_z(a)+L_z(f)+K_z(f),\qquad a\in P,\ f\in F,
\end{equation}
where $R_z=V_z|_P$, $L_z=pV_z|_F$, and $K_z=qV_z|_F$. By \eqref{eq:Pinvariant}, $R_z$ is bijective with inverse $R_{-z}$. The induced map on $M/P$ is also bijective, so $K_z^{-1}=K_{-z}$. Unitality of $V_z$ gives $K_z(q)=q$.

\emph{The finite quotient.}
We use the quadratic and cubic trace argument of Aupetit \cite{Aupetit1998}; see \cite[Section 3]{MathieuSchulz2023} for its application to invertibility-preserving maps. Here the identities follow from Liouville's theorem as above.
For $a,b\in F$, equations \eqref{eq:bilinear} and \eqref{eq:cubic} imply that
\[
 \tau\bigl(K_z(a)K_z(b^2)\bigr)=\tau(ab^2)
 =\tau\bigl(K_z(a)K_z(b)^2\bigr).
\]
Let $d=K_z(b^2)-K_z(b)^2$. Choose $a\in F$ with $K_z(a)=d^*$. Then $\tau(d^*d)=0$ for every normal tracial state. Hence $d=0$, and $K_z$ is a unital Jordan isomorphism.

\emph{The properly infinite summand.}
Lemmas~\ref{lem:squarezero} and \ref{lem:approximation} show that $R_z:P\to P$ preserves square-zero elements. The square-zero preserver theorem \cite[Theorem 4.3(i) and Remark 4.4]{Mathieu2009} applies to bounded surjective square-zero preservers from properly infinite von Neumann algebras onto unital semisimple Banach algebras. It gives
\[
 R_z(a)=u_zJ_z(a),\qquad u_z=V_z(p)\in Z(P)^{-1},
\]
where $J_z$ is a unital Jordan epimorphism. Since $R_z$ is injective and $u_z$ is invertible, $J_z$ is an isomorphism. By Lemma~\ref{lem:resolvent}, $V_z(p)$ has real spectrum in $M$, and hence in $P$. The central element $u_z$ is normal, so it is self-adjoint.

The function $z\mapsto u_z$ is constant because it is entire in norm and takes self-adjoint values. Since $u_0=p$, we obtain $u_z=p$ and $R_z$ is a unital Jordan isomorphism.

\emph{The off-diagonal component.}
Fix $f=f^*\in F$ and define
\[
 c_z=R_{-z}(L_z(f))\in P.
\]
This function is entire in norm. For every $a=a^*\in P$, Lemma~\ref{lem:resolvent} gives real spectrum for $V_z(a+f)$. Its $P$-component $R_z(a)+L_z(f)$ therefore has real spectrum in $P$. A unital Jordan isomorphism preserves spectra \cite{MathieuYoung}; applying $R_{-z}$ yields
\[
 \spec_P(a+c_z)\subset\R\qquad(a=a^*\in P).
\]
For fixed $z$, choose $a=-(c_z+c_z^*)/2$. Then $a+c_z=(c_z-c_z^*)/2$ is skew-adjoint and has real spectrum. Since it is normal, its spectrum lies in $i\R\cap\R=\{0\}$ and it vanishes. Thus $c_z=c_z^*$ for every $z$. As above, holomorphy makes $c_z$ constant, and $c_0=0$. Hence $L_z(f)=0$. Complex linearity gives $L_z=0$ on $F$.

It follows that $V_z=R_z\oplus K_z$ is a Jordan isomorphism. Its restriction $U_z$ to $C$ therefore preserves squares. Differentiating $U_z(b^2)=U_z(b)^2$ at $z=0$ gives
\[
 \delta(b^2)=\delta(b)b+b\delta(b)\qquad(b\in C).
\]
Thus $\delta$ is a bounded Jordan derivation. Sinclair's theorem \cite[Theorem 3.3]{Sinclair} implies that it is a derivation. Its exponential $U_z$ is consequently an algebra automorphism.
\end{proof}

\section{Proof of the main theorem}
We apply Proposition~\ref{prop:entire} to transported inner automorphisms and then use the bicommutant criterion of Mathieu and Sourour.

\begin{proof}[Proof of Theorem~\ref{thm:main}]
For $x\in A$, define
\[
 \delta_x=T\ad_xT^{-1},\qquad \ad_x(a)=xa-ax.
\]
Then
\begin{equation}\label{eq:transport}
 U_z=T\Ad(e^{zx})T^{-1}=e^{z\delta_x}.
\end{equation}
Here $\Ad(s)(a)=sas^{-1}$ for $s\in A^{-1}$ and $a\in A$. Equation~\eqref{eq:transport} defines a similarity family of the form used in \cite[Lemma 3]{CostaraRepovs} and \cite[Lemma 2.4]{MathieuSourour}. Each $U_z$ is unital and preserves spectral radii, since $T$, $T^{-1}$, and inner similarities have these properties. Proposition~\ref{prop:entire} implies that $\delta_x$ is a derivation of $B$ for every $x\in A$.

Let $[a,b]=ab-ba$. For $S\subset A$, let $S^{\prime}=\{a\in A:as=sa\text{ for every }s\in S\}$ be its relative commutant. Fix $h=h^*\in A$ and let
\[
 C_h=\{h\}'',\qquad E_h=T(C_h).
\]
The algebra $C_h$ is a commutative unital $C^*$-subalgebra. Indeed, both commutants are norm-closed and closed under adjoints. Since $h\in\{h\}'$, we have $C_h\subset\{h\}'$, and each element of $C_h$ commutes with every element of $\{h\}'$. The space $E_h$ is closed because $T$ is a Banach space isomorphism. Moreover,
\begin{equation}\label{eq:kernels}
 E_h=\bigcap_{x\in\{h\}'}\ker\delta_x,
\end{equation}
since $\delta_x(b)=0$ if and only if $[x,T^{-1}(b)]=0$. The kernel of a derivation is an associative subalgebra. Hence $E_h$ is a closed unital associative subalgebra of $B$.

We conclude by the commutative reduction in \cite[Proposition 3.1, (c)$\Rightarrow$(a)]{MathieuSourour}. First, $E_h$ is commutative and semisimple. If $E$ is a closed unital subalgebra of $B$ with its inherited norm, the spectral radius formula gives
\[
 r_E(w)=\lim_{n\to\infty}\norm{w^n}^{1/n}=r_B(w)\qquad(w\in E).
\]
If $w=T(c)\in E_h$ with $c\in C_h$, then $c$ is normal, so
\begin{equation}\label{eq:normradius}
 \norm{w}\le\norm{T}\norm{c}
 =\norm{T}r_A(c)=\norm{T}r_{E_h}(w).
\end{equation}
For $y,w\in E_h$, the $E_h$-valued entire function $z\mapsto e^{zy}we^{-zy}$ has constant spectral radius. Equation~\eqref{eq:normradius} bounds its norm by $\norm{T}r_{E_h}(w)$. Liouville's theorem makes it constant, and differentiation at zero gives $yw=wy$. Thus $E_h$ is commutative. Equation~\eqref{eq:normradius} also implies that $E_h$ has no nonzero quasinilpotent elements, so $\rad(E_h)=0$.

Since $r_{C_h}(c)=r_A(c)$ for $c\in C_h$, the restriction $T_h:C_h\to E_h$ is a unital surjective spectral isometry for the intrinsic spectral radii. Nagasawa's theorem \cite[Corollary 2.5]{MathieuYoung} implies that $T_h$ is multiplicative. Since $h\in C_h$,
\[
 T(h^2)=T(h)^2\qquad(h=h^*\in A).
\]
Applying this identity to $h+k$ for self-adjoint $h,k$ and subtracting the identities for $h$ and $k$ gives
\[
 T(hk+kh)=T(h)T(k)+T(k)T(h).
\]
For an arbitrary $a=h+ik\in A$, with $h,k$ self-adjoint, complex linearity now yields
\begin{align*}
 T(a^2)
 &=T(h^2-k^2+i(hk+kh))\\
 &=T(h)^2-T(k)^2+i\bigl(T(h)T(k)+T(k)T(h)\bigr)\\
 &=T(a)^2.
\end{align*}
Together with injectivity, this proves the theorem.
\end{proof}

The standard normalisation theorem gives the corresponding description when the map need not preserve the identity.

\begin{corollary}\label{cor:normalization}
Let $T:A\to B$ be a surjective complex-linear spectral isometry between unital $C^*$-algebras. Then
\[
 T(a)=uJ(a)\qquad(a\in A),
\]
where $u=T(1)$ is a central unitary in $B$ and $J:A\to B$ is a unital Jordan isomorphism.
\end{corollary}
\begin{proof}
By the normalisation theorem \cite[Corollary 2.4]{MathieuYoung},
\[
 u=T(1)\in Z(B)^{-1},\qquad
 \spec(u)\subset\{\lambda\in\C:|\lambda|=1\},
\]
and $J=u^{-1}T$ is a unital surjective spectral isometry. Centrality makes $u$ normal, and the spectral condition gives $u^*u=uu^*=1$. Theorem~\ref{thm:main} applies to $J$.
\end{proof}

\end{document}